\documentclass[11pt]{amsart}
\usepackage{amsmath}
\usepackage{amsfonts,amsthm, amscd}
\usepackage{verbatim}
\usepackage{amssymb}
\usepackage{txfonts}
\usepackage{amscd}
\usepackage{epsf}
\usepackage{stmaryrd}
\usepackage{mathrsfs}
\usepackage{hyperref}

\newtheorem{theorem}{Theorem}
\newtheorem*{theoremnn}{Theorem}

\newtheorem{corollary}[theorem]{Corollary}
\newtheorem{lemma}[theorem]{Lemma}
\newtheorem{proposition}[theorem]{Proposition}
\newtheorem{definition}[theorem]{Definition}

\newcommand{\RR}{\mathbb{R}}

\newcommand{\norm}[1]{{\|{#1}\|}}

\newcommand{\R}{\mathbb{R}}
\newcommand{\N}{\mathbb{N}}

\newcommand{\wo}{W_{0}}

\newcommand{\GX}{\hbox{$[G\curvearrowright X]$}}
\newcommand{\GY}{\hbox{$[G\curvearrowright Y]$}}
\newcommand{\ev}{\mathrm{ev}}
\newcommand{\pt}{\{*\}}

\thanks{2010 \emph{Mathematics Subject Classification.} Primary 43A07; Secondary 37A15, 58E40, 46L55}
\thanks{JB, GAN and NJW were partially supported by EPSRC grant  EP/F031947/1. 
PN was partially supported by NSF grant DMS-0900874}

\title{A homological characterization of topological amenability}

\author{Jacek Brodzki}
\address{School of Mathematics, University of Southampton, Highfield, Southampton, SO17 1SH, England}
\email{J.Brodzki@soton.ac.uk}

\author{Graham A. Niblo}
\address{School of Mathematics, University of Southampton, Highfield, Southampton, SO17 1SH, England}
\email{G.A.Niblo@soton.ac.uk}

\author{Piotr W. Nowak}
\address{Department of Mathematics, Texas A\& M University, College Station, TX 77843}
\email{pnowak@math.tamu.edu}

\author{Nick Wright}
\address{School of Mathematics, University of Southampton, Highfield, Southampton, SO17 1SH, England}
\email{N.J.Wright@soton.ac.uk}

\begin{document}

\begin{abstract}
Generalizing Block and Weinberger's characterization of amenability we introduce the notion of uniformly finite homology for a group action on a compact space and use it to give a homological characterization of topological amenability for actions. By considering the case of the natural action of $G$ on its Stone-\v Cech compactification we obtain a homological characterization of exactness of the group.
\end{abstract}

\maketitle

There are two well known homological characterizations of amenability for a countable discrete group $G$.  One, given by Johnson \cite{Johnson}, states that a group is amenable if and only if
a certain cohomology class in the first bounded cohomology $H^1_b(G,\ell_0^1(G)^{**})$ vanishes,
where $\ell_0^1(G)$ is the augmentation ideal. By contrast Block and Weinberger \cite{block-weinberger} described amenability in terms of the non-vanishing
of a homology class in the $0$-dimensional uniformly finite homology of $G$, 
$H_0^{uf}(G,\mathbb{R})$.
The relationship between these characterizations is explored in \cite{bnw}.

Amenable actions on a compact space were extensively studied by Anantharaman-Delaroche and Renault in \cite{AD-R} as a generalization of amenability which is sufficiently strong for applications and yet is exhibited by almost all known groups. A group is amenable if and only if the action on a point is amenable and it is exact  if and only if it acts amenably on its Stone-\v Cech compactification, $\beta G$, \cite{HR,guentner-kaminker,ozawa}. It is natural to consider the question of whether or not the Johnson and Block-Weinberger characterizations of amenability can be generalized to this much broader context. In particular Higson asked for such a characterization of exactness.

In \cite{bnnw} we showed how to generalize Johnson's result in terms of bounded cohomology with coefficients in a specific module $N_0(G, X)^{**}$ associated to the action. In this paper we turn our attention to the Block-Weinberger theorem, studying a related module $W_0(G,X)$ (the \emph{standard module of the action}),  and define the \emph{uniformly finite homology of the action}, ${H}^{uf}_*(G \curvearrowright X)$  as the group homology with coefficients in $W_0(G, X)^*$. The modules $N_0(G,X)^{**}$ and $W_0(G,X)^*$ should be thought of as analogues of the modules $(\ell^{\infty}(G)/\RR)^*$ and $\ell^{\infty}(G)$ respectively, which play a key role in the definition of the uniformly finite homology for groups. The two characterizations are intimately related, and we consider this relationship in section \ref{pairing}.

In the case of Block and Weinberger's uniformly finite homology the vanishing of the $0$-dimensional homology
group is equivalent to vanishing of a fundamental class $[\sum\limits_{g\in G}g]\in H_0^{uf}(G,\RR)$, however the homology group ${H}^{uf}_0(G \curvearrowright X)$ is rarely trivial even when the action is topologically non-amenable. Indeed if $X$ is a compactification of $G$ then the homology group is always non-zero, see Theorem  \ref{properaction}  below. A similar phenomenon can be observed for controlled coarse homology 
\cite{nowak-spakula}, which is another generalization of 
uniformly finite homology: only the vanishing of the fundamental class has geometric 
applications. Here we show that topological amenability is detected by a fundamental class $\GX\in {H}^{uf}_0(G \curvearrowright X)$   for the action,  and we obtain a homological characterization of topological amenability generalizing the Block-Weinberger theorem, Theorem \ref{amenable}, which may be summarized as follows:

\begin{theoremnn}\label{homological}
Let $G$ be a finitely generated group acting by homeomorphisms on a compact Hausdorff 
topological space $X$. The action of $G$ on $X$ is topologically amenable if and only if 
the fundamental class $\GX$ is non-zero in ${H}^{uf}_0(G \curvearrowright X)$. 
\end{theoremnn}

When the space $X$ is a point, the uniformly finite homology of the action 
$H_n^{uf}(G\curvearrowright X)$ reduces to $H_n^{uf}(G,\mathbb{R})$,  the uniformly
finite homology of $G$ with real coefficients \cite{block-weinberger}, recovering the characterization
proved by Block and Weinberger.

\section{The uniformly finite homology of an action}\label{sect1}

Let $G$ be a group generated by a \emph{finite} set $S=S^{-1}$, 
acting by homeomorphisms on a compact Hausdorff space $X$.

The space $C(X, \ell^1(G))$ of continuous $\ell^1(G)$ valued functions on $X$ is equipped with the $\sup-\ell^1$ norm
$$\Vert \xi\Vert=\sup\limits_{x\in X}\sum_{g\in G}\vert \xi(x)(g)|.$$

The summation map on $\ell^1(G)$ induces a continuous map $\sigma:C(X, \ell^1(G))\rightarrow C(X)$, where $C(X)$ is equipped with the $\ell^\infty$ norm. The space $N_0(G, X)$ is defined to be the pre-image $\sigma^{-1}(0)$ which we identify as $C(X, \ell^1_0(G))$, while, identifying $\mathbb R$ with the constant functions on $X$ we define $W_0(G,X)$ to be the subspace $N_0(G,X)+\mathbb R = \sigma^{-1}(\mathbb R)$. Restricting $\sigma$ to the subspace $W_0(G,X)$ we can regard it as a map $W_0(G,X)\to \R$, and with this convention we may regard $\sigma$ as an element of the dual space $W_0(G,X)^*$.

Given an element $\xi\in C(X, \ell^1(G))$ we obtain a family of functions $\xi_g\in C(X)$ indexed by the elements of $G$ by setting $\xi_g(x)=\xi(x)(g)$.

In this notation, the Banach space $C(X, \ell^1(G))$ is equipped with a natural action of $G$,
$$(g\cdot \xi)_h=g*\xi_{g^{-1}h},$$
for each $g,h\in G$, where $*$ denotes the translation action of $G$ on $C(X)$:
$g*f(x)=f(g^{-1}x)$ for $f\in C(X)$. We note that with these actions on $C(X, \ell^1(G))$ and $C(X)$, the map $\sigma$ is equivariant which implies that $N_0, W_0$ are $G$-invariant subspaces.

\begin{definition}[\cite{bnnw}]
We call $\wo(G,X)$, with the above action of $G$, the \emph{standard module of the action} of $G$ on $X$.
\end{definition}

We have the following short exact sequence of $G$-modules:
\[
0\longrightarrow N_0(G,X) \xrightarrow{i} W_0(G,X) \xrightarrow{\sigma} \mathbb{R} \longrightarrow 0.
\]

It is also worth pointing out that when $X$ is a point we have 
$W_0(G,X)=\ell^1(G)$ and $N_0(G,X)=\ell^1_0(G)$. The above 
modules and decompositions were introduced, with a slightly different but equivalent description, in \cite{bnnw} for a compact $X$ 
and in \cite{douglas-nowak} in the case when $X=\beta G$, the Stone-\v{C}ech compactification of $G$.

Recall that if $V$ is a $G$-module then $V^*$ is a also a $G$-module with the action of $G$
given by $(g\psi)(\xi)=\psi(g^{-1}\xi)$ for $\psi\in V^*$ and $\xi\in V$. With this definition we introduce the notion of uniformly finite homology for a group action.

\begin{definition}
Let $G$ be a finitely generated group acting by homeomorphisms on a compact space $X$.
We define the \emph{uniformly finite homology of the action} by setting 
$$H_n^{uf}(G\curvearrowright X)=H_n(G,W_0(G,X)^*),$$
for every $n\ge 0$, where $H_n$ denotes group homology.
\end{definition}

A certain homology class in the uniformly finite homology of the action will be of 
particular importance to us.

\begin{definition}
Let $G$ act by homeomorphisms on a compact space $X$. The \emph{fundamental class of 
the action}, denoted $\GX$, is the homology class in ${H}^{uf}_0(G \curvearrowright X)$ 
represented by $\sigma$. 
\end{definition}

As noted above, when $X$ is a point we have $W_0(G,X)=\ell^1(G)$, so $W_0(G,X)^*=\ell^{\infty}(G)$,
 $\GX=[\sum\limits_{g\in G}g]$, and 
$$H_0^{uf}(G,\RR)\simeq H_0(G,\ell^{\infty}(G))\simeq H^{uf}_0(G,W_0(G,\operatorname{pt})^*)={H}^{uf}_0(G \curvearrowright \operatorname{pt}).$$

Consider the dual of the short exact sequence of coefficients above:
\[
0\rightarrow \mathbb R^*\xrightarrow{\sigma^*} W_0(G,X)^*\rightarrow N_0(G,X)^*\rightarrow 0.
\]
The map $\sigma$ is always split as a vector space map, and hence its dual $\sigma^*$ is also split. We now consider the question of when we can split the map $\sigma^*$ equivariantly. Identifying $\R^*$ with $\R$, the map $\sigma^*$ takes $1$ to $\sigma$, hence the condition that $\mu:W_0(G,X)^*\to \mathbb R$ splits $\sigma^*$ is the condition $\mu(\sigma)=1$. Hence a $G$-equivariant splitting of $\sigma^*$ can be regarded as a $G$-invariant functional $\mu\in W_0(G, X)^{**}$ such that $\mu(\sigma)=1$. But this is precisely an invariant mean for the action as described in \cite[definition 13]{bnnw}, so we obtain:

\begin{lemma}
Let $G$ be a group acting by homeomorphisms on a compact Hausdorff space $X$. Then the action is topologically amenable if and only if there is a $G$-equivariant splitting of the map $\sigma^*$ in the short exact sequence 

\begin{equation*}
0\rightarrow \mathbb R^*\xrightarrow{\sigma^*} W_0(G,X)^*\rightarrow N_0(G,X)\rightarrow 0.
\end{equation*}
\end{lemma}

Applying this lemma to the long exact sequence in group homology arising from the short exact sequence above we obtain:

\begin{corollary}\label{longexactcoeffs}
If the group $G$ acts topologically amenably on the compact Hausdorff space $X$, then for each $n$ there is a short exact sequence

\[
0\rightarrow H_n(G, \mathbb R)\rightarrow H_n(G, W_0(G,X)^*)\rightarrow H_n(G, N_0(G, X)^*)\rightarrow 0,
\]
mapping the fundamental class $[1]\in H_0(G, \R)$ to the fundamental class $\GX$ of the action. This gives us an isomorphism 

\[
H_n^{uf}(G\curvearrowright X)\cong H_n(G,  \mathbb R) \oplus H_n(G, N_0(G,X)^*).
\]
\end{corollary}

In Theorem \ref{amenable} we characterize topological amenability in terms of the $0$-dimensional homology. In particular when the action is not topologically amenable we will show (Corollary \ref{decomp}) that $H_0^{uf}(G\curvearrowright X)$ is isomorphic to $H_0(G, N_0(G,X)^*)$.
 
\section{Non-vanishing elements in $H_0^{uf}(G\curvearrowright X)$}

Unlike the Block-Weinberger case, vanishing of the fundamental class does not in general imply the vanishing of ${H}^{uf}_0(G \curvearrowright X)$.

\begin{theorem}
\label{properaction}
Let $X$ be a compact $G$ space containing an open $G$-invariant subspace $U$ on which $G$ acts properly. Then ${H}^{uf}_0(G \curvearrowright X)$ is non-zero.
In particular ${H}^{uf}_0(G \curvearrowright \overline{G})$ is non-zero for any compactification $\overline{G}$ of $G$.
\end{theorem}

\begin{proof}
If $G$ is finite, and the action of $G$ on $X$ is trivial, then ${H}^{uf}_0(G \curvearrowright X)=W_0(G,X)^*$ which is non-zero.

Otherwise we may assume that the action of $G$ on $U$ is non-trivial, replacing $U$ with $X$ if $G$ is finite. Thus we may pick a point $x_0$ in $U$, and $x_1=g_1x_0$ in $Gx_0$ with $x_0\neq x_1$. Let $f\in C(X)$ be a positive function of norm 1, with $f(x_0)=1$ and with the support $K$ of $f$ contained in $U\setminus\{x_1\}$. By construction $x_0\notin g_1^{-1}K$.

Define $\xi\in W_0(G,X)$ by $\xi_e=f, \xi_{g_1}=-f,$ and $\xi_{g}=0$ for $g\neq e,g_1$. We note that $\xi$ is in $W_0(G,X)$ as required, indeed it is in $N_0(G,X)$, since $\sum_{g\in G}\xi_g$ is identically zero. We now form the sequence
$$\xi^n=\sum_{k\in G}\phi_n(k) k\cdot \xi,\text{ where }\phi_n(k)=\max\left\{\frac{n-d(e,k)}{n},0\right\}.$$
If $\xi^n_g(x)$ is non-zero then $x$ is in $gK$ or $gg_1^{-1}K$. By properness of the action there are only finitely many $h\in G$ such that $hK$ meets $K$. Let $N$ be the number of such $h$. If $x\in hK$, then $x\in gK \cup gg_1^{-1}K$ for at most $2N$ values of $g$, hence for each $x\in X$, the set of $g$ with $\xi^n_g(x)\neq 0$ has cardinality at most $2N$. Since $|\xi^n_g(x)|\leq 2$ for each $g, n, x$ it follows that $\|\xi^n\|\leq 4N$ for all $n$.

For $s\in S$ consider
$$\xi^n-s\cdot\xi^n=\sum_{g\in G}\phi_n(g) (g\cdot \xi-sg\cdot \xi)=\sum_{g\in G}(\phi_n(g)-\phi_n(s^{-1}g)) g\cdot \xi.$$
Since $|(g\cdot\xi)_h(x)|\leq 1$ for all $x$ and $|\phi_n(g)-\phi_n(s^{-1}g)|\leq \frac 1n$ it follows that $|(\xi^n-s\cdot\xi^n)_h(x)|\leq \frac 1n$ for all $h,x$. On the other hand, for a given $x$, $(\xi^n-s\cdot\xi^n)_h(x)$ is non-zero for at most $4N$ values of $h$, hence $\|\xi^n-s\cdot\xi^n\|\leq \frac {4N}{n}$. We thus have a sequence $\xi^n$ in $W_0(G,X)$ with $\|\delta \xi^n\|\to 0$. It follows that if $\zeta$ is a weak-* limit point of $\xi^n$ in $W_0(G,X)^{**}$ then $\delta^{**}\zeta=0$, so $\zeta$ is a cocycle defining a class $[\zeta]$ in $H^0(G,W_0(G,X)^{**})$.

Let $\ev_{e,x_0}\in W_0(G,X)^*$ be the evaluation functional $\eta\mapsto \eta_e(x_0)$, and consider the homology class $[\ev_{e,x_0}]\in H^{uf}_0(G \curvearrowright X)$. We have
$$\ev_{e,x_0}(\xi^n)=\xi^n_e(x_0)=\phi_n(e) (e\cdot\xi)_{e}(x_0)+\phi_n(g_1^{-1}) (g_1^{-1}\cdot\xi)_e(x_0)$$
since the other terms in the sum vanish. The first term is $\phi_n(e)f(x_0)=1$, while $(g_1^{-1}\cdot\xi)_e(x_0)=(g_1^{-1}*\xi_{g_1})(x_0)=0$ since $x_0$ is not in $g_1^{-1}K$. Thus $\ev_{e,x_0}(\xi^n)=1$ for all $n$. It follows that the pairing of $[\ev_{e,x_0}]$ with $[\zeta]$ is 1, hence $[\ev_{e,x_0}]$ is a non-trivial element of $H^{uf}_0(G \curvearrowright X)$.
\end{proof}

We remark that there is a surjection from $H^{uf}_0(G \curvearrowright X)$ onto $H_0(G,N_0(G,X)^*)$, induced by the surjection $W_0(G,X)^*\to N_0(G,X)^*$, and the non-trivial elements constructed in the proposition remain non-trivial after applying this map.

\section{Characterizing amenability}

We recall the definition of a (topologically) amenable action.

\begin{definition}\label{action} Let $G$ be a group acting by homeomorphisms on a compact Hausdorff space. The action of $G$ on $X$ is said to be \emph{topologically amenable} if 
 there exists a sequence of 
elements $\xi^n\in W_{00}(G,X)$ such that  
\begin{enumerate}
\item $\xi^n_g\ge0$ in $C(X)$ for every $n\in \N$ and $g\in G$, 
\item $\sigma(\xi^n)=1$ for every $n$,
\item $\sup_{s\in S}\Vert \xi^n-s\cdot \xi^n\Vert\to 0$.
\end{enumerate}
\end{definition}

Universality of the Stone-\v Cech compactification leads to the observation that a group acts amenably on some compact space if and only if it acts amenably on $\beta G$, which is equivalent to exactness. Amenable actions on compact spaces (lying between the point and $\beta G$) form a spectrum of generalized amenability properties interpolating between amenability and exactness. We will return to this point later.

Now consider the coboundary map

\[
\begin{CD}
W_0(G,X) @>\delta>>\left(\displaystyle{\bigoplus_{s\in S}}W_0(G,X)\right)_{\infty},
\end{CD}
\]
where 
$$(\delta\xi)_s=\xi-s\cdot \xi,$$
for $\xi\in W_0(G,X)$, where the (finite) direct sum is equipped with a supremum norm. 
The operator $\delta$ is clearly bounded.
Since $S$ is finite the dual of $\delta$ is
\[
\begin{CD}
W_0(G,X)^*@>\delta^*>>\left(\displaystyle{\bigoplus_{s\in S}}W_0(G,X)^*\right)_{1},
\end{CD}
\]
where the direct sum is equipped with an $\ell^1$-norm and the adjoint map is given by
$$\delta^*\psi=\sum_{s\in S} \psi_s-s^{-1}\cdot \psi_s.$$
The functional $\sigma$ can be used to detect amenability of the action.

\begin{theorem}\label{theorem :  characterization by sigma in the image}
Let $G$ be a finitely generated group acting on a compact space $X$ by homeomorphisms. 
The following conditions are equivalent:
\begin{enumerate}
\item the action of $G$ on $X$ is topologically amenable,
\item $\sigma\notin\overline{\operatorname{Image}(\delta^*)}^{\Vert \cdot \Vert}$,
\item $\sigma\notin \operatorname{Image}(\delta^*)$,
\end{enumerate}
\end{theorem}
\begin{proof}

$(1)\implies (2)$. Assume first that the action is amenable. Take $\mu$ to be the weak-* limit of 
a convergent subnet of $\xi_{\beta}$
as in the definition of amenable actions. Then 
$$\mu(\sigma)=\lim_{\beta} \sigma( \xi_{\beta}) =1,$$
and in particular $\sigma$ is not in the kernel of $\mu$. On the other hand
$$|\mu(\delta^*\psi)|=\lim_{\beta} |\delta^*\psi (\xi_{\beta})|=\lim_{\beta} |\psi(\delta \xi_{\beta})|\le 
\lim_{\beta}\left(\Vert \psi\Vert\ \sup_{s\in S}\Vert \xi_{\beta}-s\cdot \xi_{\beta}\Vert\right)=0,$$
for every $\psi\in \bigoplus_{s\in S} W_0(G,X)^*$.
Thus 
$$\operatorname{Image}(\delta^*)\subseteq \ker \mu.$$
Since $\ker \mu$ is norm-closed, we conclude
$$\overline{\operatorname{Image}(\delta^*)}^{\Vert \cdot \Vert}\subseteq \ker \mu.$$
Thus $\sigma \notin \overline{\operatorname{Image}(\delta*)}^{\Vert \cdot \Vert}$ and 
(2) follows.

$(2)\implies (3)$ is obvious.

To prove $(3)\implies (1)$ we suppose there exists a constant $D>0$ such that 
\begin{flalign*}
(\dagger)&&\Vert \delta \xi\Vert &\ge D \vert \sigma(\xi)\vert&&
\end{flalign*}
for all $\xi$, and seek a contradiction. Consider a functional $\psi:\delta(W_0(G,X))\to \RR$, defined by
$$\psi(\delta\xi)=\sigma(\xi).$$
This is well defined, since $\delta: W_0(G,X)\to \bigoplus_{s\in S} W_0(G,X)$ is injective.
By inequality ($\dagger$), $\psi$ is continuous on $\delta(W_0(G,X))$
and, by the Hahn-Banach theorem, we can extend it to a continuous functional $\Psi$
on $\bigoplus_{s\in S} W_0(G,X)$. By definition, for $\xi\in  W_0(G,X)$ we have
$$
[\delta^*(\Psi)](\xi)=\Psi(\delta\xi)=\psi(\delta\xi)=\sigma(\xi),$$
hence $\sigma$ is in the image of $\delta^*$, contradicting (3).

It follows that there is no $D>0$ such that inequality ($\dagger$) holds for all $\xi\in W_0(G,X)$, hence there exists a sequence $\xi^n\in W_0(G,X)$ such that $\sigma(\xi^n)=1$ for all $n$, and $\Vert\delta \xi^n\Vert\to 0$. Since $W_{00}(G,X)$ is dense in $W_0(G,X)$, we may assume without loss of generality that $\xi^n\in W_{00}(G,X)$, and applying the standard normalization argument we deduce that the action is amenable.
\end{proof}

We are now in the position to prove the main theorem, which is stated here in a more general form. 
The \emph{reduced homology} $\overline{H}^{uf}_n(G \curvearrowright X) =\overline{H}_n(G,W_0(G,X)^*)$ 
in the statement is defined, as in the context of $L^2$-(co)homology, by taking the closure of the images in the chain complex.

\begin{theorem}\label{amenable}
Let $G$ be a finitely generated group acting by homeomorphisms on a compact space $X$.
The following conditions are equivalent
\begin{enumerate}
\item the action of $G$ on $X$ is topologically amenable,
\item $\GX\neq 0$ in $\overline{H}^{uf}_0(G \curvearrowright X)$,
\item \label{vanishingclass} $\GX\neq 0$ in ${H}^{uf}_0(G \curvearrowright X)$,
\item the map $(i^*)_*:H_0^{uf}(G\curvearrowright X)\to H_0(G, N_0(G,X)^*)$ is not injective,
\item the map $(i^*)_*:H_1^{uf}(G\curvearrowright X)\to H_1(G, N_0(G,X)^*)$ is surjective.
\end{enumerate}
\end{theorem}

\begin{proof}
The equivalence 
(1)$\Longleftrightarrow$(2)$\Longleftrightarrow$(3) follows from 
Theorem  \ref{theorem :  characterization by sigma in the image}.
Indeed,  we have $H_0(G,M)=M_G$,  where $M_G$ is 
the coinvariant module, namely the quotient of $M$ by the
module generated by elements of the form $g\cdot m-m$. Since $G$ is finitely generated
it is enough to consider only sums of elements of the form $s\cdot m-m$, where $s$ are the generators.
Indeed, if $g=s_1s_2\dots s_n$ for $s_i\in S$, we can write 
$$g\cdot m-m=\left(\sum_{i=1}^{n-1} s_i\cdot m_i-m_i\right)+s_n\cdot m-m, $$
where $m_i=(s_{i+1}\dots s_n)\cdot m$ for $i\leq n$.
Hence $W_0(G,X)^*_G$ is exactly the quotient $W_0(G,X)^*$ by the image of 
$\delta^*$.

As in the proof of Corollary \ref{longexactcoeffs} the short exact sequence of coefficients yields a long exact sequence which terminates as

\[
\rightarrow H_0(G, \mathbb R^*)\xrightarrow{\sigma^*} H_0(G, W_0(G,X)^*)\xrightarrow{i^*} H_0(G, N_0(G, X)^*)\rightarrow 0,
\]
and in which the fundamental class $[1]\in  H_0(G, \mathbb R^*)$ maps to the class $\GX$. Thus $\GX\neq 0$ if and only if the map $\sigma^*$ is non-zero, or equivalently the kernel of $i^*$ is non-zero. Thus it follows that (3) is equivalent to (4).

Also by exactness
of the sequence $\GX\neq 0$ if and only if $[1]$ is not in the image of the connecting map, or equivalently the connecting map is zero, and we obtain the equivalence of (3) and (5).
\end{proof}

Combining this with Corollary \ref{longexactcoeffs} we obtain:

\begin{corollary}\label{decomp}
Let $G$ be a group acting by homeomorphisms on a compact Hausdorff topological space $X$. 

\[H_0^{uf}(G \curvearrowright X)\cong 
\begin{cases}
H_0(G, \mathbb R)\oplus H_0(G, N_0(G,X)^*) &\text{ when  the action is amenable,}\\
H_0(G, N_0(G,X)^*) & \text{ when  the action is not amenable.}\\
\end{cases}
\]

\end{corollary}

\section{Functoriality}

We return to the remark that we made earlier that the actions of $G$ on compact spaces form a spectrum, with the single point at one end of the spectrum and the Stone-\v Cech compatification of $G$ at the other end. We can make sense of this statement homologically as follows.

Suppose that $G$ is a finitely generated group acting by homeomorphisms on two compact spaces $X, Y$. Given a continuous, equivariant map $X \rightarrow Y$ of compact $G$-spaces we obtain induced continuous maps $f^*: C(Y,\ell^1(G))\to C(X,\ell^1(G))$ and $f^*: C(Y)\to C(X)$ defined by $f^*(\xi)=\xi\circ f$. Let $\sigma_X:C(X,\ell^1(G))\to C(X),\sigma_Y:C(Y,\ell^1(G))\to C(Y)$ denote the summation maps. Summation is compatible with the pull-backs in the sense that $\sigma_X\circ f^*=f^*\circ\sigma_Y$, hence $f^*$ restricts to maps $W_0(G,Y)\to W_0(G,X)$ and $N_0(G,Y)\to N_0(G,X)$. Note that equivariance of $f$ implies equivariance of $f^*$.

Let $\xi\in C(Y,\ell^1(G))$ we have
$$\norm{f^*\xi}=\sup_{x\in X}\sum_{g\in G} |\xi_g(f(x))|\leq \sup_{y\in Y}\sum_{g\in G} |\xi_g(y)|=\norm{\xi}$$
so when $f$ is surjective, we have equality, and $f^*$ is an isometry onto its image. Dualising the restriction of $f^*$ to $W_0(G,Y)\to W_0(G,X)$ we obtain a continuous linear map which we denote by $f_*:W_0(G,X)^*\rightarrow W_0(G, Y)^*$. Equivariance of this map follows from equivariance of $f^*$.

As the map $f_*$ is equivariant, it induces a map on group homology (also denoted $f_*$):
\[
f_*:H^{uf}_n(G \curvearrowright X)\to H^{uf}_n(G \curvearrowright Y)
\]

In the special case that $f$ is surjective, as $f^*$ is an isometry onto its image it follows that $f_*$ is surjective, so we obtain a short exact sequence of $G$-modules

\[
0\rightarrow W_0(G,f)^*\rightarrow W_0(G,X)^*\xrightarrow {f_*}W_0(G, Y)^*\rightarrow 0
\]
where $W_0(G,f)$ denotes the quotient space $W_0(G,X)/f^*W_0(G, Y)$.

This induces a long exact sequence in group homology from which we extract the following fragment.
\[
\cdots \rightarrow H^{uf}_0(G \curvearrowright X)\xrightarrow{f_*} H^{uf}_0(G \curvearrowright Y)\rightarrow 0.
\]
Thus surjectivity of $f$ implies surjectivity of the map $f_*$ on homology in dimension 0.

In general, whether $f$ is surjective or not, the fundamental class $\GY$ is in the image of $f_*$. Specifically we have $f_*\GX=\GY$ which follows from the identity $\sigma_X\circ f^*=f^*\circ\sigma_Y$.

It follows that if $\GY$ is non-trivial then so is $\GX$, recovering the statement that if the action on $Y$ is topologically amenable then so is the action on $X$.

Now suppose that $X$ is an arbitrary compact space on which $G$ acts by homeomorphisms so by universality there  are equivariant continuous  maps 

\[
\beta G\rightarrow X\rightarrow \pt.
\]
It follows that if $G$ is amenable then the action on $X$ is topologically amenable. On the other hand if the action on $X$ is topologically amenable then the action on $\beta G$ is also topologically amenable,  hence $G$ is exact. Hence we recover two well known facts.

Consider again the general situation of a continuous $G$-map $f:X\rightarrow Y$. We have seen that topological amenability automatically transfers from $Y$ to $X$, but in general it does not transfer in the opposite direction. In order to transfer it from $X$ to $Y$ we need to place additional constraints on the map $f$.

\begin{definition}
Let $G$ be a group and $X,Y$ be compact Hausdorff topological spaces on which $G$ acts by homeomorphisms. A continuous $G$-equivariant map $f:X\rightarrow Y$  induces a $G-C(Y)$-module structure on $C(X)$ by pullback. The map $f$ is said to be \emph{amenable} if there is a bounded $C(Y)$-linear $G$-equivariant map $\mu: C(X) \rightarrow C(Y)$ 
with $\mu({\bf 1}_X)={\bf 1}_Y$.
\end{definition}
Amenability of the map $f$ implies that $f$ is surjective, hence $f^*$ is topologically injective and $\mu$ is a splitting of $f^*$. 

When $G$ is the trivial group this reduces to the classical definition of an amenable map, while if $Y$ is a point then the map $X\rightarrow Y$ is amenable if and only if the action of $G$ on $X$ is co-amenable.

\begin{proposition}
Let $G$ be a group and $X,Y$ be compact Hausdorff topological spaces on which $G$ acts by homeomorphisms. Let $f:X\rightarrow Y$ be an amenable continuous  $G$-equivariant map. If the action of $G$ on $X$ is topologically amenable then so is the action of $G$ on $Y$.
\end{proposition}

\begin{proof}
We use the isomorphism between the space  $C(X,\ell^1(G))$ and the completed injective tensor product $ C(X)\, \widehat{\otimes}_\epsilon \,\ell^1(G)$
(see, e.g., \cite[Theorem 44.1]{Treves}) to identify
 $W_0(G,X)$ as a subspace of $C(X)\,\widehat{\otimes}_\epsilon \,\ell^1(G)$ and $W_0(G,Y)$ as a subspace of $C(Y)\,\widehat{\otimes}_\epsilon \, \ell^1(G)$. Since $f$ is amenable we have a  $G$-equivariant splitting $\mu:C(X) \rightarrow C(Y)$ of the map $f^*$, giving a map $\mu\otimes_\epsilon 1:C(X)\,\widehat{\otimes}_\epsilon \, \ell^1(G)\rightarrow C(Y)\,\widehat{\otimes}_\epsilon \, \ell^1(G)$. This restricts to a map $W_0(G, X)\rightarrow W_0(G,Y)$ since $\mu$ takes constant functions on $X$ to constant functions on $Y$. 
 
 The corresponding dual  map $W_0(G,Y)^*\rightarrow W_0(G,X)^*$ induces a map on homology that, abusing notation, 
 we will denote $\mu^*:H_0(G\curvearrowright Y)\rightarrow H_0(G\curvearrowright X)$. By construction this splits the map $f_*:H_0(G\curvearrowright X)\rightarrow H_0(G\curvearrowright Y)$, and since $\mu({\bf 1}_X )= {\bf 1}_Y$, $\mu^*(\GY)=\GX$. It follows that if the fundamental class $\GX$ is not trivial then neither is $\GY$, and so topological amenability of the action on $X$ implies topological amenability for the action on $Y$ as required.
\end{proof}

\section{The interaction between uniformly finite homology and bounded cohomology}\label{pairing}

We conclude with some remarks concerning the interaction of the uniformly finite homology of an action and the
bounded cohomology with coefficients introduced in \cite{bnnw}. These illuminate the special role played by the Johnson class in $H_b^1(G, N_0(G,X)^{**}$ and the fundamental class in $H_0^{uf}(G\curvearrowright X)$ and extend the results in \cite{bnw} which considered the special case of the action of $G$ on a point.

In \cite{bnnw} we showed that topological
amenability of the action is encoded by triviality of an element $[J]$ in $H^1_b(G, N_0(G,X)^{**})$, which we call the Johnson class for the action. This class is the image of the class $[1]\in H^0_b(G, \RR)$ under the connecting map arising from the short exact sequence of coefficients

\[ 0\rightarrow N_0(G,X)^{**}\rightarrow W_0(G,X)^{**}\rightarrow \RR\rightarrow 0
\]

\noindent which is dual to the short exact sequence appearing in the proof of Theorem \ref{amenable}. 

By applying the forgetful functor from bounded to ordinary cohomology, we obtain a pairing of $H^1_b(G, N_0(G,X)^{**})$ with $H_1(G, N_0(G,X)^{*})$, and clearly if the Johnson class $[J]$ is trivial then its pairing with any $[c]\in H_1(G,N_0^*)$ is zero.

Now suppose that every $[c]\in H_1(G,N_0(G,X)^*)$ pairs trivially with
the Johnson class. Since the Johnson class $[J]$ is obtained by applying the connecting map to the generator $[1]$ of $H^0_b(G, \RR)=\RR$, pairing $[J]$ with $[c]\in H_1(G,N_0(G,X)^*)$ is
the same as pairing $[1]$ with the image of $[c]$ under the connecting map in homology. As this pairing (between $H^0(G,\R)=H^0_b(G,\R)$ and $H_0(G,\R)$) is faithful, it follows that the image of $[c]$ under the connecting map is trivial for all $[c]$, so the connecting map is zero, which we have already noted is equivalent to amenability of the action. Thus in the case when the group is non-amenable, the non-triviality of the Johnson element must be detected by the pairing.

On the other hand, we can run a similar argument in the opposite
direction: if pairing $\GX$ with every element $[\phi]\in H^0_b(G,W_0(G,X)^
{**})$ we get zero, then since $\GX=(\sigma^*)_*[1]$, we have that the
pairing of $(\sigma^{**})_*[\phi]\in H^0_b(G,\R)$ with $[1]\in H_0(G,\R)$ is trivial, whence $(\sigma^{**})_*[\phi]=0$ (again by faithfulness of the pairing). Thus, by exactness, the connecting map on cohomology is injective and the Johnson class is non-trivial. So when the action is amenable, (and hence the Johnson class is trivial), non-triviality of $\GX$ must be detected by the pairing.

\bibliographystyle{plain}

\end{document}